\documentclass[12pt]{article}
\usepackage{indentfirst, latexsym, bm, amssymb}

\begin{document}

\newtheorem{theorem}{Theorem}[section]
\newtheorem{corollary}[theorem]{Corollary}
\newtheorem{definition}[theorem]{Definition}
\newtheorem{conjecture}[theorem]{Conjecture}
\newtheorem{question}[theorem]{Question}
\newtheorem{lemma}[theorem]{Lemma}
\newtheorem{proposition}[theorem]{Proposition}
\newtheorem{example}[theorem]{Example}
\newtheorem{problem}[theorem]{Problem}
\newenvironment{proof}{\noindent {\bf
Proof.}}{\rule{3mm}{3mm}\par\medskip}
\newcommand{\remark}{\medskip\par\noindent {\bf Remark.~~}}
\newcommand{\pp}{{\it p.}}
\newcommand{\de}{\em}

\newcommand{\JEC}{{\it Europ. J. Combinatorics},  }
\newcommand{\JCTB}{{\it J. Combin. Theory Ser. B.}, }
\newcommand{\JCT}{{\it J. Combin. Theory}, }
\newcommand{\JGT}{{\it J. Graph Theory}, }
\newcommand{\ComHung}{{\it Combinatorica}, }
\newcommand{\DM}{{\it Discrete Math.}, }
\newcommand{\ARS}{{\it Ars Combin.}, }
\newcommand{\SIAMDM}{{\it SIAM J. Discrete Math.}, }
\newcommand{\SIAMADM}{{\it SIAM J. Algebraic Discrete Methods}, }
\newcommand{\SIAMC}{{\it SIAM J. Comput.}, }
\newcommand{\ConAMS}{{\it Contemp. Math. AMS}, }
\newcommand{\TransAMS}{{\it Trans. Amer. Math. Soc.}, }
\newcommand{\AnDM}{{\it Ann. Discrete Math.}, }
\newcommand{\NBS}{{\it J. Res. Nat. Bur. Standards} {\rm B}, }
\newcommand{\ConNum}{{\it Congr. Numer.}, }
\newcommand{\CJM}{{\it Canad. J. Math.}, }
\newcommand{\JLMS}{{\it J. London Math. Soc.}, }
\newcommand{\PLMS}{{\it Proc. London Math. Soc.}, }
\newcommand{\PAMS}{{\it Proc. Amer. Math. Soc.}, }
\newcommand{\JCMCC}{{\it J. Combin. Math. Combin. Comput.}, }
\newcommand{\GC}{{\it Graphs Combin.}, }

\title{ Faber-Krahn type inequality for unicyclic graphs \thanks{
This work is supported by National Natural Science Foundation of
China (No:10971137), the National Basic Research Program (973) of
China (No.2006CB805900),  and a grant of Science and Technology
Commission of Shanghai Municipality (STCSM, No: 09XD1402500).  }}
\author{ Guang-Jun  Zhang, Jie Zhang, Xiao-Dong Zhang\thanks{Corresponding  author ({\it E-mail address:}
xiaodong@sjtu.edu.cn)}
\\
{\small Department of Mathematics},
{\small Shanghai Jiao Tong University} \\
{\small  800 Dongchuan road, Shanghai, 200240,  P.R. China}\\
In  Memory of Professor Ky Fan
 }
\maketitle
 \begin{abstract}
  The Faber-Krahn inequality
 states that the ball has  minimal first Dirichlet eigenvalue
  among all bounded domains with the  fixed volume in
$\mathbb{R}^n$.  In this paper,  we investigate  the similar
inequality for unicyclic graphs. The results show that the
Faber-Krahn type inequality also holds for  unicyclic graphs with a
given graphic unicyclic
  degree sequence  with minor conditions.

   \end{abstract}


\vskip 0.5cm

\section{Introduction}
The Faber-Krahn inequality which is a well-known result on the
Riemannian manifolds  states that the ball has  minimal first
Dirichlet eigenvalue among all bounded domains with the same volume
in $\mathbb{R}^n$ (with the standard Euclidean metric). It has been
first proved independently by Faber  and Krahn  for the
$\mathbb{R}^2$. A proof of the generalized version can be found
 in \cite{chavel1984}. Since the graph Laplacian can be
regarded as the discrete analog of the continuous
Laplace-Beltrami-operator on manifolds,  the Faber-Krahn inequality
for  graphs has received more and more attentions. Friedman
\cite{friedman1993} introduced the idea of a ``graph with boundary
'' and formulated the Dirichlet eigenvalue problem for graphs.
Leydold
 \cite{leydold1997} and \cite{leydold2002} proved that the
Faber-Krahn type inequality  held for regular trees and gave a
complete characterization of all extremal trees.  In 1998, Pruss
\cite{pruss1998} proposed the following question: which classes of
graphs  has the Faber-Krahn property? Recently,
B${\i}$y${\i}$ko\u{g}lu and Leydold \cite{biykoglu2007} proved that
 the Faber-Krahn inequality also held for trees with the same
degree sequence. The vertices of the unique extremal tree possesses
a spiral like ordering, i.e., ball approximations. Moreover, they
proposed the following problem.
\begin{problem}(\cite{biykoglu2007})
 Give a characterization of all graphs in a given class $\mathcal{C}$
with the Faber-Krahn property, i.e., characterize those graphs in
$\mathcal{C}$ which have minimal first Dirichlet eigenvalue for a
given ``volume''.
\end{problem}
 Motivated by the above question and results, we  investigate
the Faber-Krahn type inequality  for  unicyclic graphs with a given
 degree sequence.
 Before stating our main results, we
introduce some necessary notations.

 In this paper, we only consider simple and undirected graphs.
 Let $G=(V(G), E(G))$ be a graph of order $n$ with vertex
set $V(G)$ and edge set $E(G)$. Let $A(G)=(a_{uv})$ be the adjacency
matrix of $G$ with $a_{uv}=1$ for $u$ adjacent to $v$ and $0$ for
otherwise. The Laplacian matrix of $G$ is defined as
$L(G)=D(G)-A(G)$, where $d(v)$ is the degree of vertex $v$ and
$D(G)=diag(d(v),\ v\in V(G))$ is the degree diagonal matrix of $G$.
A connected graph is called to be {\it unicyclic} if the number of
vertices is equal to the number of edges. Then  a unicyclic graph
has the only one cycle.  A
 positive integer sequence $\pi
=(d_0,d_1,\cdots,d_{n-1})$ is called a {\it graphic unicyclic degree
sequence} if there exists  a unicyclic graph $G$ whose degree
sequence  is $\pi$. For a given graphic unicyclic degree sequence
$\pi=(d_0, d_1, \cdots,d_{n-1}), $
 denote by  $\mathcal{U}_{\pi}$  the set of all
 unicyclic graphs  with the degree sequence $\pi$.
The main results of this paper
can be stated as follows:
\begin{theorem}\label{theorem1.1}  For a given graphic unicyclic degree sequence
$\pi=(d_0, d_1, \cdots,d_{n-1}), $ with $3\le d_0\le\dots\le d_k $
and $d_{k+1}=\cdots=d_{n-1}=1$,
  let $G=(V_0\cup \partial V, E_0\cup \partial E)$ be a  graph with
  the Faber-Krahn property in $\mathcal{U}_{\pi}.$ Then $G$ has an
 SLO-ordering (see in section 3) consistent with the first eigenfunction $f$ of
 $G$ in such a way that $ v \prec u$ implies $f(v) \geq f(u)$.
\end{theorem}

\begin{theorem}
\label{theorem1.2} For a given graphic unicyclic degree sequence
$\pi=(d_0, d_1, \cdots,d_{n-1}), $
 with $3\le d_0\le\dots\le d_k $ and $d_{k+1}=\cdots=d_{n-1}=1$,
Then $U_{\pi}^{\ast}$ (see in section 4) is the only one   graph
with the Faber-Krahn property in $\mathcal{U}_{\pi}$, which can be
regarded as  ball approximation.
\end{theorem}

\begin{remark}
If the frequency of 2 in $\pi$ is at least one, then
Theorems~\ref{theorem1.1} and ~\ref{theorem1.2} may not hold (see in
section 5).
\end{remark}

The rest of this paper is organized as follows: In section 2, we
recall some notations of the
 first Dirichlet eigenvalue of a graph with boundary.
The proof of Theorems~\ref{theorem1.1} and \ref{theorem1.2} will be
presented in sections 3 and 4, respectively. In section 5,  some
examples and  remarks  explain that Theorems~\ref{theorem1.1} and
\ref{theorem1.2} do not generally hold  for a given graphic
unicyclic degree sequence with the frequency of 2 being at least
one.

\section{\large\bf{The first Dirichlet eigenvalue}}
A graph with boundary $G=(V_0\cup \partial V$,
 $E_0\cup \partial E)$ consists of a set of interior vertices
 $V_0$,
boundary vertices $\partial V$, interior edges $E_0$ that connect
interior vertices, and boundary edges $\partial E$ that join
interior vertices with boundary vertices  (for example, see
\cite{chung1997} or \cite{friedman1993}). Throughout this paper we
always assume that the degree of any boundary vertex is 1 and the
degree of any interior vertex is at least 2.

  A real number $\lambda$ is called a {\it Dirichlet eigenvalue} of $G$
  if there exists a function $f \neq 0$
such that
they satisfy the Dirichlet eigenvalue problem:
 \begin{displaymath}
\left \{ \begin{array}{ll}
L(G)f(u)=\lambda f(u) & u \in V_0;\\
f(u)=0 & u \in \partial V.\\
\end{array} \right.
\end{displaymath}
The function $f$ is called an {\it eigenfunction} corresponding to
$\lambda$.
  \begin{definition}\label{definition1.1}
 (\cite{biykoglu2007}). A graph with boundary has the Faber-Krahn
property if it has minimal first Dirichlet eigenvalue among all
graphs with the same ``volume'' in a particular graph class.
\end{definition}

In this paper, we use a given graphic unicyclic  degree sequence as
the volume and the unicyclic graphs with this volume as the graph
class. The Rayleigh quotient of the Laplace operator $L$ on
real-valued functions $f$ on $V(G)$ is
$$R_{G}(f)=\frac{<Lf,f>}{<f,f>}=\frac{\sum\limits_{uv\in E(G)}(f(u)-f(v))^{2}}{\sum\limits_{v\in V(G)}f^{2}(v)}.$$
 If $\lambda(G)$ is the first Dirichlet eigenvalue of
 $G$, then
$$\lambda(G)= \min\limits_{f\in S}R_{G}(f)=\min\limits_{f\in
S}\frac{<Lf,f>}{<f,f>},$$  where $S$ is the set of all real-valued
functions on $V(G)$ with the constraint $f|_{\partial V} = 0$.
Moreover, if $R_{G}(f)=\lambda(G)$ for a function $f \in S$, then
$f$ is an eigenfunction of $\lambda(G)$ (see \cite{biykoglu2007} or \cite{friedman1993}).

\section{\large\bf{The proof of Theorem~\ref{theorem1.1}}}
In order to prove Theorem~\ref{theorem1.1}, we need some notations
and lemmas.
  B${\i}$y${\i}$ko\u{g}lu and Leydold  \cite{biykoglu2007}
 extended the concept of an SLO-ordering for describing the trees with
the Faber-Krahn property, which is introduced by Pruss (see
\cite{pruss1998}). The notation of an SLO-ordering may be extended
for
 any connected graphs.

\begin{definition}\label{definition3.1} (\cite{biykoglu2007})Let $G=(V_0\cup \partial V, E_0\cup
\partial E)$ be a connected graph with root $v_0$. Then a
well-ordering $\prec$ of the vertices is called spiral-like
(SLO-ordering for short) if
 the following holds for all vertices $u, v, x, y \in V(G)$:

 (1) $ v\prec u $ implies $h(v)\leq h(u)$, where $h(v)$ denotes the distance between $v$ and $v_0$;

 (2) let $uv \in E(G)$, $xy \in E(G), uy \notin E(G)$,
 $xv \notin E(G)$
with $h(u)=h(v)-1$ and $h(x)=h(y)-1$. If $u \prec x$, then $v \prec y$ ;

 (3) if $ v \prec u$ and $ v\in \partial V$, then $u \in \partial V$.
\end{definition}

Clearly,  if $G$ is a tree, an SLO-ordering of $G$  is consistent
with the definition of an SLO-ordering in \cite{biykoglu2007}.
Moreover, if there exists a positive integer $r$ such that  the
number of
 vertices $v$ with $h(v)=i+1$ is not less than the number of  vertices
$v$ with  $h(v)=i$ for $i=1, \cdots, r-1$,  and $h(v) \in \{r,
r+1\}$ for any boundary vertex $v \in
\partial V$,  $G$ is called a {\it ball approximation}. The
graph $G$ in Fig.~1  has an SLO-ordering and is a ball
approximation.

\setlength{\unitlength}{1mm}
\begin{picture}(120, 40)
\centering \put(60,35){\circle*{1}}
   \put(40,30){\circle*{1}}
\put(60,30){\circle*{1}} \put(80,30){\circle*{1}}
\put(40,25){\circle*{1}} \put(55,25){\circle*{1}}
\put(70,25){\circle*{1}} \put(80,25){\circle*{1}}
\put(85,25){\circle*{1}} \put(35,20){\circle*{1}}
\put(45,20){\circle*{1}} \put(50,20){\circle*{1}}
\put(60,20){\circle*{1}} \thicklines

\qbezier(40, 30)(65, 25)(80, 30) \put(60,35){\line(-4,-1){20}}
\put(60,35){\line(0,-1){5}} \put(60,35){\line(4,-1){20}}
\put(40,30){\line(0,-1){5}} \put(60,30){\line(-1,-1){10}}
\put(60,30){\line(2,-1){10}} \put(80,30){\line(0,-1){5}}
\put(40,25){\line(-1,-1){5}} \put(40,25){\line(1,-1){5}}
\put(55,25){\line(-1,-1){5}} \put(55,25){\line(1,-1){5}}
\put(80,30){\line(1,-1){5}} \put(60,36){$v_0$} \put(36,30){$v_1$}
\put(63,30){$v_2$} \put(81,30){$v_3$} \put(36,25){$v_4$}
\put(57,25){$v_5$} \put(68,22){$v_6$} \put(78,22){$v_7$}
\put(83,22){$v_8$} \put(33,17){$v_9$} \put(42,17){$v_{10}$}
\put(48,17){$v_{11}$} \put(58,17){$v_{12}$}

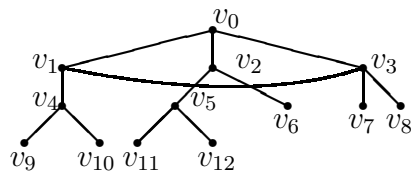
\put(10,5){\small{Fig.1 $\displaystyle \hbox{ $G$ with degree
sequence} \  \pi=(3, 3,3,3,3,4,1, 1,1,1, 1,1, 1). $}}
\end{picture}
\begin{lemma}\label{lemma3.2}
(\cite{friedman1993}) Let $G=(V_0 \cup
\partial V, E_0 \cup \partial E)$ be a connected graph with boundary. Then

(1) $\lambda(G) $ is a positive simple eigenvalue;

(2) An eigenfunction $f$ of the eigenvalue $\lambda(G)$ is either
positive or negative on all interior vertices of $G$.
\end{lemma}
Clearly, there exists only one eigenfunction $f$  of $\lambda(G)$
that satisfies  $f(v)>0$ for $v \in V_0$,
 $f(u)=0$ for $u \in \partial V$
 and $||f||$=1 by Lemma~\ref{lemma3.2}. Moreover, $f$ is called
 the {\it first eigenfunction} of $G$.
 Let $G-uv$ denote the graph obtained from $G$ by  deleting an edge $uv$ in
$G$ and $G+uv$ denote the graph obtained from $G$ by adding an edge
$uv$. The following result is from \cite{biykoglu2007}.

\begin{lemma}\label{lemma3.3}(\cite{biykoglu2007})
Let $G=(V_0\cup \partial V, E_0\cup \partial E)$ be a connected
graph. Suppose that  there exist four vertices $u_1$, $v_1$, $v_2
\in V_0$ and $u_2 \in V_0\cup \partial V$ with $u_1v_1$, $u_2v_2 \in
E_0\cup
\partial E$ and $u_1u_2, v_1v_2\notin E_0\cup \partial E.$ Let
$G^{\prime}=G-u_1v_1-u_2 v_2+u_1 u_2+v_1v_2$ and $ f $ be the first
eigenfunction of $G$.  If $f(v_1)\geq f(u_2)$ and $f(v_2)\geq
f(u_1)$, then
$$\lambda (G^{\prime})\leq \lambda (G).$$
Moreover, inequality  is strict if one of the two inequalities is
strict.
\end{lemma}

The following corollary can be directly deduced from
Lemma~\ref{lemma3.3}

\begin{corollary}\label{corollary3.4} For a given graphic unicyclic degree sequence
$\pi=(d_0, d_1, \cdots,d_{n-1}), $
  let $G=(V_0\cup \partial V, E_0\cup \partial
E)$ be a  graph with the Faber-Krahn property
 in $\mathcal{U}_{\pi}$. Suppose that there exist four vertices $u$, $v$, $x \in V_0$ and
  $y \in V_0\cup \partial V$ with
$uv$, $xy \in E_0\cup \partial E  $ and $ux, vy\notin E_0\cup
\partial E.$ Let $f$ be the first eigenfunction of $G$ and
$G^{\prime}=G- uv-xy+ux+ vy.$ If $G^{\prime} \in \mathcal{U}_{\pi}$,
then the following holds:

(1) if $f(u)=f(y)$, then $f(v)=f(x)$;

(2) if $f(u)>f(y)$, then $f(v)>f(x)$.
\end{corollary}

\begin{lemma}\label{lemma3.5} For a given graphic unicyclic degree sequence
$\pi=(d_0, d_1, \cdots,d_{n-1})$
 , let $G=(V_0\cup \partial V, E_0\cup \partial E)$ be a
 graph
 with the Faber-Krahn property  in $ \mathcal{U}_{\pi}$.
 If  $C$ is a  cycle of $G$
 and $f$ is the first eigenfunction
of $G$, then $f(x) > f(u)$ for any $x \in V(C)$ and $u \in (V_0\cup
\partial V) \setminus V(C).$
\end{lemma}

\begin{proof} Suppose that there are two vertices $x \in V(C)$ and $u \in (V_0\cup
\partial V) \setminus V(C)$
such that $f(x) \leq f(u).$ Then $f(u)\geq f(x)>0$ since $x$ is an
interior vertex. So $u$ is  an interior vertex by
Lemma~\ref{lemma3.2}. Let $uw $ be the first edge of the shortest
path from vertex $u$ to  cycle $C.$ Since $u \notin V(C)$ and $G$ is
unicyclic, $uw$ is a cut edge of $G$. Then $G-uw$ has  the exact two
connected components $G_1$ containing $C$ and $G_2$ containing $u.$
Moreover, $G_2$ is a tree and  contains all neighbor vertices except
$w$. Hence there exists a path $P=u u_1 \cdots u_m $ in $G_2$ with
$m\geq 1$ and $u_m \in
\partial V$.  Since $G$  is unicyclic, $u$ is adjacent to at most one vertex in $V(C)$.
Hence there exists a vertex $y\in V(C)$ with  $xy \in E(C)$ and $uy
\notin E(G)$. Since $V(C)\subseteq V(G_1) $ and $V(P)\subseteq
V(G_2)$, we have $V(P) \cap V(C)=\phi$ and $xu_i, yu_i \notin E(G)$
for all $1 \leq i \leq m$. Let $G_1=G-xy-uu_1+yu+xu_1.$  Then $G_1
\in \mathcal{U}_{\pi}$ and
  $f (u_1)>f(y)\ge\min\{f(x), f(y)\}>0$ by Corollary~\ref{corollary3.4}.
Further $G_2=G-xy-u_1 u_2+yu_2+xu_1.$ Then $G_2 \in
\mathcal{U}_{\pi}$
 and $f(u_2)>f(x)\ge \min\{f(x), f(y)\}>0$ by Corollary~\ref{corollary3.4}.
By repeating this procedure,  we have $f(u_i)> f(x)\ge\min\{f(x),
f(y)\}>0$ if $i$ is even and  $f(u_i) >f(y)\ge \min\{f(x), f(y)\}>0$
if $i$ is odd, where $i=1, \cdots, m.$ Hence at last, we have
$f(u_m)> \min\{f(x), f(y)\}>0$. But $f(u_m)=0$  since  $u_m$ is a
boundary vertex. It is a contradiction. Therefore, the assertion
holds.
\end{proof}

\begin{lemma}\label{lemma3.6} For a given graphic unicyclic degree sequence
$\pi=(d_0, d_1, \cdots,d_{n-1})$
 with $3\le d_0\le\dots\le d_k $ and $d_{k+1}=\cdots=d_{n-1}=1$,
let $G=(V_0 \cup \partial V, E_0\cup \partial E)$ be a graph
 with the Faber-Krahn property  in $ \mathcal{U}_{\pi}$
 and $f$ be the first eigenfunction of $G$.
If there exists a set $V^{\prime}=\{v_0,v_1, v_2\}$ such that
$f(v_0) \geq f(v_1)\geq f(v_2)\geq f(x)$ for $x \in (V_0 \cup
\partial V) \setminus V^{\prime}$, then  the induced subgraph
$G[V^{\prime}]$  by $V^{\prime}$  is the only one  cycle of $G$.
\end{lemma}
\begin{proof} Since $G$ is unicyclic, let $C$ be the only one  cycle in $G$.
 By Lemma~\ref{lemma3.5}, it is easy to see that
  $v_0, v_1, v_2 \in V(C).$  we now prove that $G[V^{\prime}]$ is a triangle.
  If $v_0v_1 \notin
E(G),$ then there are two vertices  $x \in V(C)$ and $y \notin V(C)$
 such that $v_0x\in E(G)$ and $v_1y \in E(G)$.
 Let $G_1=G-v_0x-v_1y+v_0v_1+xy.$ Clearly, $G_1 \in  \mathcal{U}_{\pi}.$
 Moreover, $f(v_1) \geq f(x)$ and $f(v_0)> f(y)$ by Lemma~\ref{lemma3.5}.
Then  $\lambda (G_1)< \lambda (G)$ by Lemma~\ref{lemma3.3}, which is
a contradiction with $G $ having the Faber-Krahn property in $
\mathcal{U}_{\pi}.$ Similarly, we have $v_0v_2 \in E(G).$ Suppose
now $v_1v_2 \notin E(G).$ Then there is a vertex $u \in V(C)$ such
that $u \neq v_0$ and $v_1 u \in E(G).$ Since $v_2 \in V_0,$ there
is a vertex $z \notin V(C)$ such that  $v_2z \in E(G).$
 Let $G_2=G-v_1u-v_2z+v_1v_2+uz.$ Note that $f(v_2) \geq f(u)$ and
  $f(v_1)> f(z)$ by Lemma~\ref{lemma3.5}.  Then $G_2 \in  \mathcal{U}_{\pi}$
and $\lambda (G_1)< \lambda (G)$ by Lemma~\ref{lemma3.3}, which is
impossible. So $v_1v_2 \in E(G).$ The proof is completed.
\end{proof}

{\bf Proof of Theorem~\ref{theorem1.1}}:
  Without loss  of generality,  assume $V(G)$ $=\{v_0$, $v_1,$ $\cdots,$ $v_{n-1}\}$ such that $f(v_0)\geq f(v_{1})\geq \cdots \geq f(v_{n-1}).$
 Then we have  $v_0v_1,$ $v_0v_2,$ $v_1v_2 \in E(G)$ by Lemma~\ref{lemma3.6}. Clearly, $v_0$ is an interior vertex. Let $v_0$ be the root of $G.$
 Suppose $h(G)=\max\limits_{v \in V(G)}h(v)$. Let $W_i=\{v \in V(G)| h(v)=i\}$ and $|W_i|=n_i$ for $0 \leq i \leq h(G)$.
For convenience of our proof,
we relabel the vertices of $G$.
Let $v_0=v_{0, 1}.$ Then $W_0=\{v_{0, 1}\}.$ Clearly, $n_1=d(v_0)$. The
vertices in $W_1$ are relabeled as $v_{1, 1},$ $v_{1, 2},$ $\cdots,$ $v_{1, n_1}$
such that $f(v_{1, 1})$ $\geq f({v_{1, 2}})$ $\geq \cdots$  $\geq f(v_{1, n_1}).$
Assume that the vertices in $W_t$ have been already relabeled as $v_{t, 1},
{v_{t, 2}}, \cdots, v_{t, n_t}$. Then the vertices in $W_{t+1}$ can be
relabeled as $v_{t+1, 1}, v_{t+1, 2}, \cdots, v_{t+1, n_{t+1}}$ such
that they satisfy the following conditions: if $v_{t, k}v_{t+1, i},$
$v_{t, k}v_{t+1, j}\in E(G)$ and $i<j$, then  $f(v_{t+1, i}) \ge
f(v_{t+1, j})$; if $v_{t, k}v_{t+1, i},  v_{t, l}v_{t+1, j}\in E(G)$
and $k<l$, then $i<j$.

{\bf Claim :} $f(v_{t, 1})\geq f(v_{t, 2})\geq \cdots \geq f(v_{t, n_t})\geq  f(v_{{t+1}, 1})$ for $0 \leq t \leq h(G).$

We will prove that the Claim  holds by  induction. Clearly,
the Claim  holds for $t=0.$
Assume now that the Claim  holds for $t=s-1$.  In the following we
prove that the Claim holds for $t=s$. If there are two
vertices $v_{s, i},v_{s, j}\in W_s$ with $i < j$ and $f(v_{s, i})< f(v_{s, j})$, then
there exist two vertices
 $v_{s-1, k}, v_{s-1, l}\in W_{s-1}$ with $k < l$ such that
 $v_{s-1, k}v_{s, i}, v_{s-1, l}v_{s, j} \in E(G)$.
 By the induction hypothesis, $f(v_{s-1, k})\geq f(v_{s-1, l}).$
 Let $G_1 = G-v_{s-1, k}v_{s, i}-v_{s-1, l}v_{s, j}+v_{s-1, k}v_{s, j}+v_{s-1, l}v_{s, i}.$
Clearly, $G_1\in \mathcal{U}_{\pi}.$ By Lemma~\ref{lemma3.3}, we have
$\lambda(G_1)< \lambda(G),$ which is a contradiction to our assumption that $G$ has the  Faber-Krahn property
 in $ \mathcal{U}_{\pi}$.  So $f(v_{s, i}) \geq
f(v_{s, j}).$ Assume now $f(v_{s, n_s})<  f(v_{{s+1}, 1}).$ Note
that $d(v_0) \geq 3.$ It is easy to see that $v_{s, n_s}v_{s-1,
n_{s-1}}$,  $v_{{s}, 1}v_{{s+1}, 1} \in E(G).$ By the induction
hypothesis, $f(v_{s-1, n_{s-1}}) \geq f(v_{s, 1}).$ Let $G_2=G-v_{s,
n_s}v_{s-1, n_{s-1}}-v_{{s}, 1}v_{{s+1}, 1} +v_{s, n_s}v_{{s},
1}+v_{s-1, n_{s-1}}v_{{s+1}, 1}.$ Then  there exists a  $G_2\in
\mathcal{U}_{\pi}$  such that $\lambda(G_2)< \lambda(G)$ by
Lemma~\ref{lemma3.3}, which is also a contradiction. So  the Claim
holds. Therefore we finish our proof.$\blacksquare$

\section{\large\bf{The proof of Theorem~\ref{theorem1.2} }}

In order to prove Theorem~\ref{theorem1.2}, we need the following
lemmas
\begin{lemma}\label{lemma4.1}For a given graphic unicyclic degree sequence
$\pi=(d_0, d_1, \cdots,d_{n-1})$, let $G=(V_0\cup \partial V,E_0\cup
\partial E) \in \mathcal{U }_{\pi}$ with the first eigenfunction $f$.
If there exist two vertices $v_1, v_2 \in V_0$ such that $u_{t}v_{1}
\in E(G)$, $u_{t}v_2 \notin E(G)$ for $t=1,2,\cdots,p\leq d(v_1)-2$,
let  $G^{\prime}$ be the graph obtained from $G$ by deleting the $p$
edges  $u_1 v_1, \cdots, u_pv_1$ and adding the $p$ edges $u_1 v_2,
\cdots, u_p v_2$. If $G^{\prime}$ is connected and  $f(v_1)\geq
f(v_2)\geq f(u_t)$ for $t=1, 2, \cdots, p,$  then $G^{\prime}$ and
$G$ have the same boundary vertices,  and
$$\lambda(G^{\prime})\leq \lambda(G).$$
Moreover, the inequality is strict if there exists $u_s$ with $1
\leq s \leq p$ such that $f(v_1)>f(u_s)$.
\end{lemma}

 \begin{proof} Clearly, $G^{\prime} \in \mathcal{U}_{\pi}$  and $G^{\prime}$ and $G$ have the same boundary
vertices. Further
\begin{eqnarray}
\lambda(G^{\prime})-\lambda(G) &\leq& R_{G^{\prime}}(f) - R_{G}(f)\nonumber\\
&=& \sum \limits_{i=1}^{t}(f(v_2)-f(u_i))^2 - \sum \limits_{i=1}^{t}(f(v_1)-f(u_i))^2 \nonumber\\
\ & \leq & 0.\nonumber
\end{eqnarray}
Assume that there exists a vertex $u_s$  such that $f(v_1)>f(u_s)$.
If $\lambda(G^{\prime})=\lambda(G)$, then $f$ also must be an
eigenfunction of $\lambda(G^{\prime}).$ By
\begin{eqnarray}
\lambda(G^{\prime})f(v_1)=L(G^{\prime})f(v_1)
&=&\sum\limits_{z, v_1z\in E(G^{\prime})}(f(v_1)-f(z))\nonumber\\
&=&\lambda(G)f(v_1)=L(G)f(v_1)\nonumber\\
&=&\sum\limits_{z, v_1z\in E(G^{\prime})}(f(v_1)-f(z))+\sum\limits_{i=1}^{t}(f(v_1)-f(u_i)),\nonumber
\end{eqnarray}
we have $f(v_1)=f(u_t)$ for $t=1, 2, \cdots, p$. This is a
contradiction to  $f(v_1)>f(u_s)$. So the assertion holds.
\end{proof}

 Let $G$ be a graph with  root $v_0$ and  $u$
   be adjacent to $v$. If  $h(u)=h(v)+1$, then we
call $u$ a {\it child } of $v$ and $v$ a {\it parent }
of $u.$ If $h(u)=h(v)$, we call $u$ a {\it brother} of $v.$  With
this notation, we have following:

\begin{lemma}\label{lemma4.2} For a given graphic unicyclic degree sequence
$\pi=(d_0, d_1, \cdots,d_{n-1})$
 with $3\le d_0\le\dots\le d_k $ and $d_{k+1}=\cdots=d_{n-1}=1$,
 let $G=(V_0\cup \partial V, E_0\cup \partial E)$ be a
 graph with   the Faber-Krahn property in
$\mathcal{U}_{\pi}$. Then the  SLO-ordering of $G$ induced
by the first eigenfunction $f$ of $\lambda(G)$ has the
following property: `` for every interior vertex $v$ without brother,
 there exists a child  $u$ of $v$ such
that $f(u) < f(v)$".
\end{lemma}

\begin{proof}
By Lemma~\ref{lemma3.6}  and Theorem~\ref{theorem1.1},
 $G$ has an SLO-ordering $v_0 \prec v_1 \prec
\cdots \prec
 v_{n-1}$ such that $f(v_0)\geq f(v_1)\geq \cdots \geq f(v_{n-1})$ and the
 only one
cycle  $v_{0}v_{1}v_{2}$.
 If $v=v_0$ and  $f(x)=f(v)$ for any child $x$ of $v$, then by $L(G)f=\lambda(G)f$,  we have
$$\lambda(G)f(v_0)=d(v_0)f(v_0)-\sum_{wv_0\in E(G)}f(w)=0,$$
which implies $\lambda(G)=0$. This is a contradiction to the
statement (1) of Lemma~\ref{lemma3.2}. If  $v\neq v_0$,  let $w$ be
the parent of $v$ and $u_1, u_2, \cdots, u_t$ be  all children of
$v.$ Then by the proof of Theorem~\ref{theorem1.1}, $f(w) \geq f(v)
\geq f(u_j)$ for $j=1, 2, \cdots, t$. If $f(u_j) = f(v)$ for $j=1,
2, \cdots, t$, we have
\begin{eqnarray}
\lambda(G)f(v)=L(G)f(v) & = & d(v)f(v)-f(w)-\sum\limits_{j=1}^{t}f(u_j) \nonumber\\
\  & =& f(v)-f(w)\leq 0,\nonumber
\end{eqnarray}
which also is a contradiction to Lemma~\ref{lemma3.2}. Hence the
assertion holds.
\end{proof}

For a given unicyclic degree sequence $\pi=(d_{0}, d_{1}, \cdots,
d_{n-1})$ with $3 \leq d_{0}\leq d_{1}\leq \cdots \leq d_{k-1}$ and
 $d_k=d_{k+1}=\cdots =d_{n-1}=1$, where $n\ge 3$ and $2<k<n-1$.
 We now construct a unicyclic graph $U_{\pi}^{\ast}$ with
degree sequence $\pi$ as follows. Select a vertex $v_{0,1}$ as a
root and begin with $v_{0, 1}$ of the zero-th layer. Let $s_1=d_0$
and select $s_1$ vertices ${v_{1, 1}=v_1, v_{1, 2}=v_2, \cdots,
v_{1,s_1}}=v_{s_1}$ of the first layer such that they are adjacent
to $v_{0,1}$ and $v_{1,1}$ is adjacent to $v_{1,2}$.  Next we
construct the second layer as follows. Let
$s_2=\sum\limits_{i=1}^{s_1}d_i-s_1-2$ and select $s_{2}$ vertices
$v_{2,1}, v_{2,2}, \cdots, v_{2, s_2}$ such that $v_{1,1}$ is
adjacent to $v_{2,1},\cdots, v_{2, d_1-2}$; $v_{1,2}$ is adjacent to
$v_{2, d_1-1}, \cdots, v_{2, d_1+d_2-4}$, $v_{1,3}$ is adjacent to
$v_{2, d_1+d_2-3}, \cdots, v_{2, d_1+d_2+d_3-5}$, $\cdots, $ $
v_{1,j}$ is adjacent to $v_{2, d_1+\cdots d_{j-1}-j}, \cdots, v_{2,
d_1+\cdots+d_j-j-2}$, $\cdots, $ $ v_{1, s_1}$ is adjacent to $v_{2,
d_1+\cdots+d_{s_1-1}-s_1}, \cdots, v_{2,
d_1+\cdots+d_{s_1}-s_1-2}=v_{2, s_2}$. In general, assume that all
vertices of the $t$-st layer have been constructed and are denoted
by  ${v_{t,1}, v_{t,2}, \cdots, v_{t,s_{t}}}$. We construct all the
vertices of the $(t+1)$-st layer by the induction. Let
$s_{t+1}=d_{s_1+\cdots+ s_{t-1}+1}+\cdots +d_{s_1+\cdots+s_t}-s_t$
and select $s_{t+1}$ vertices $v_{{t+1}, 1}, v_{{t+1}, 2}, \cdots,
v_{t+1, s_{t+1}}$ of the $(t+1)$st layer such that $v_{t,1}$ is
adjacent to $v_{t+1,1},$ $\cdots,$ $v_{t+1,
d_{s_1+\cdots+s_{t-1}+1}-1}$, $\cdots,$ $ v_{t, s_t}$ is adjacent to
$v_{t+1, s_{t+1}-d_{s_1+\cdots+s_t}+2}, \cdots, v_{t+1, s_{t+1}}$.
 In this way, we obtain the unique unicyclic graph
$U_{\pi}^{\ast}$ with degree sequence $\pi$ such that the root
$v_{0, 1}$ has minimum degree in all interior vertices.
\begin{example} Let $\pi =(3, 3, 3, 4, 4, 5, 1, 1, 1, 1, 1, 1, 1, 1, 1,
1)$. Then $U_{\pi}^{\ast}$ is as follows in Fig.2:
\end{example}

\setlength{\unitlength}{1mm}
\begin{picture}(120, 40)
\centering \put(60,35){\circle*{1}} \put(40,30){\circle*{1}}
\put(65,30){\circle*{1}} \put(80,30){\circle*{1}}
\put(40,25){\circle*{1}} \put(65,25){\circle*{1}}
\put(75,25){\circle*{1}} \put(80,25){\circle*{1}}
\put(85,25){\circle*{1}} \put(35,20){\circle*{1}}
\put(40,20){\circle*{1}} \put(45,20){\circle*{1}}
\put(55,20){\circle*{1}} \put(60,20){\circle*{1}}
\put(65,20){\circle*{1}} \put(70,20){\circle*{1}}

\thicklines \put(40,30){\line(1,0){25}} \put(65,30){\line(-1,0){15}}
\put(60,35){\line(-4,-1){20}} \put(60,35){\line(1,-1){5}}
\put(60,35){\line(4,-1){20}} \put(40,30){\line(0,-1){5}}
\put(65,30){\line(0,-1){5}} \put(80,30){\line(0,-1){5}}
\put(40,25){\line(-1,-1){5}} \put(40,25){\line(1,-1){5}}
\put(65,25){\line(-1,-1){5}} \put(65,25){\line(1,-1){5}}

\put(80,30){\line(-1,-1){5}} \put(80,30){\line(1,-1){5}}
\put(40,25){\line(0,-1){5}} \put(40,25){\line(1,-1){5}}
\put(65,25){\line(-2,-1){10}} \put(65,25){\line(0,-1){5}}
\put(60,36){$v_{0,1}$} \put(34,30){$v_{1,1}$} \put(66,30){$v_{1,2}$}
\put(81,30){$v_{1,3}$} \put(34,25){$v_{2,1}$} \put(58,25){$v_{2,2}$}
\put(71,23){$v_{2,3}$} \put(78,23){$v_{2,4}$} \put(86,23){$v_{2,5}$}
\put(30,17){$v_{3,1}$} \put(36,17){$v_{3,2}$} \put(43,17){$v_{3,3}$}
\put(52,17){$v_{3,4}$} \put(58,17){$v_{3,5}$} \put(64,17){$v_{3,6}$}
\put(70,17){$v_{3,7}$} \put(30,5){\small{Fig.2\ \ \  $\displaystyle
U_{\pi} ^{\ast} \hbox{ with degree sequence}\ \pi $}}
\end{picture}

{\bf Proof of Theorem~\ref{theorem1.2}}: Let $G$ be a  graph with the Faber-Krahn property in $ \mathcal{U}_{\pi}$ and $f$ be the first
eigenfunction of
 $G$. By Lemma~\ref{lemma3.6}  and Theorem~\ref{theorem1.1},
 $G$ has an SLO-ordering $v_0 \prec v_1 \prec
\cdots \prec
 v_{n-1}$ such that $f(v_0)\geq f(v_1)\geq \cdots \geq f(v_{n-1})$ and the
 only one
cycle  $v_{0}v_{1}v_{2}$. Since $f$ is the first
eigenfunction of
 $G$,
  $v_0,$ $v_1,$ $\cdots,$ $v_{k-1}$ are all interior vertices of $G$ by Lemma~\ref{lemma3.2}.

Claim: $d(v_0) \leq d(v_1)\leq \cdots
\leq d(v_{k-1})$.

Assume that the Claim does not hold.
Then there exists the
smallest non-negative integer $t \in \{0, 1, \cdots ,k-2\}$  such that $d(v_t)>
d(v_{t+1})$.
If $t \geq 3,$  then $v_t$ has $d(v_t)-1$ children, one parent and no brother.
Let
$w_1,w_2,\cdots,w_{d(v_t)-1}$ be all the children of
 $v_t$ with $f(w_{i})\geq f(w_{i+1})$ for $1 \leq i \leq d(v_t)-2$. Then we have  $f(v_t)\geq
f(v_{t+1}) \geq f(w_{d(v_{t+1})})\geq f(w_{d(v_{t+1})+1})\geq \cdots \geq
f(w_{d(v_t)-1})$ by  Theorem~\ref{theorem1.1}.  Further $f(v_{t}) > f(w_{d(v_t)-1})$ by Lemma~\ref{lemma4.2}.
Let $G_1$ be the graph obtained from $G$ by deleting the edges $v_{t}
w_s$ and
adding the edges $v_{t+1}w_s$  for $s=d(v_{t+1}),d(v_{t+1})+1,\cdots,d(v_t)-1$.
Clearly, $G_1 \in \mathcal{U}_{\pi}$ and $\lambda(G_1) < \lambda(G)$ by Lemma~\ref{lemma4.1}.
This is a contradiction to our assumption that $G$ has the Faber-Krahn property
 in $ \mathcal{U}_{\pi}$.
If $t=0,$ then
$v_0$ has $d(v_0)$ children  and no parent.
If $t=1$ or $2,$  $v_t$ has $d(v_t)-2$ children, one parent and one brother.
Note that there are only two vertices $v_1, v_2$ having brother. Then for any
$u \in \{v_0, v_1, \cdots, v_{k-1}\},$
there is a child $x$ of $u$ such that $f(x) < f(u)$ by Lemma~\ref{lemma3.5} and  Lemma~\ref{lemma4.2}.
By applying the similar argument as above, our hypothesis is also impossible for $t \leq 2.$
Thus the Claim  holds. Then
by the Calim, we have $d(v_i)=d_i$ for $0 \leq i \leq n-1.$
So  $G$ is isomorphic to $U_{\pi}^{\ast}$.
The proof is completed.  $\blacksquare$

 From the proof of Theorem~\ref{theorem1.2},  we can get the following

  \begin{corollary}\label{corollary4.4}  For a given graphic unicyclic degree sequence
  $ \pi = (d_0, d_1, \cdots, d_{n-1})$  with $3 \leq d_0 \leq d_1\leq
\cdots \leq d_{k-1}$ and $d_k=d_{k+1}=\cdots =d_{n-1}=1$, let $G$ be
the graph with the Faber-Krahn property in $\mathcal{U}_{\pi}$.
  Then $G$ has an SLO-ordering  $v_0 \prec v_1
\prec \cdots \prec
 v_{n-1}$ such that $d(v_i)=d_i$ for $i=0,1,\cdots,n-1$.
\end{corollary}

\section{\large\bf{Examples and Remarks  }}
B${\i}$y${\i}$ko\u{g}lu and  Leydold  \cite{biykoglu2007}
characterized all extremal graphs with the Faber-Krahn property
among all trees with any tree degree sequence $\pi$. Moreover, the
unique extremal graph can be regarded as a ball approximation. In
this paper,  For a given graphic unicyclic degree sequence
  $ \pi = (d_0, d_1, \cdots, d_{n-1})$  with $3 \leq d_0 \leq d_1\leq
\cdots \leq d_{k-1}$ and $d_k=d_{k+1}=\cdots =d_{n-1}=1$, we
characterized all extremal graphs  with the Faber-Krahn property
among all unicyclic graphs in $\mathcal{U}_{\pi}$. The unique
extremal graph can also be regarded as a ball approximation. It is
natural to ask that the assertion still holds for other graphic
unicyclic degree sequence $\pi$? In the following, we present some
observation on graphic unicyclic degree sequence $\pi$ with the
frequency of 2 being at least one.

\begin{example}\label{example5.1} Let
 $G_1$ and  $G_2$  be the following two graphs with degree sequence $\pi_1
=(2, 2, 2, 3, 3, 4, $ $ 5, 1, 1, 1, 1, 1, 1, 1)$:
\end{example}
\setlength{\unitlength}{1mm}
\begin{picture}(120, 45)
\centering \put(40,40){\circle*{1}} \put(35,35){\circle*{1}}
\put(50,35){\circle*{1}} \put(50,30){\circle*{1}}
\put(50,25){\circle*{1}} \put(45,20){\circle*{1}}
\put(60,20){\circle*{1}} \put(40,15){\circle*{1}}
\put(45,15){\circle*{1}} \put(50,15){\circle*{1}}
\put(55,15){\circle*{1}} \put(60,15){\circle*{1}}
\put(65,15){\circle*{1}} \put(70,15){\circle*{1}} \thicklines
\put(40,40){\line(-1,-1){5}} \put(40,40){\line(2,-1){10}}
\put(35,35){\line(1,0){15}} \put(50,35){\line(0,-1){5}}
\put(50,30){\line(0,-1){5}} \put(50,25){\line(-1,-1){5}}
\put(50,25){\line(2,-1){10}} \put(45,20){\line(-1,-1){5}}
\put(45,20){\line(0,-1){5}} \put(45,20){\line(1,-1){5}}
\put(45,20){\line(-1,-1){5}} \put(60,20){\line(-1,-1){5}}
\put(60,20){\line(0,-1){5}} \put(60,20){\line(1,-1){5}}
\put(60,20){\line(2,-1){10}} \put(40,42){$v_{0}$}
\put(30,35){$v_{1}$} \put(52,35){$v_{2}$} \put(52,29){$v_{3}$}
\put(52,25){$v_{4}$} \put(39,20){$v_{5}$} \put(62,20){$v_{6}$}
\put(37,12){$v_{7}$} \put(42,12){$v_{8}$} \put(48,12){$v_{9}$}
\put(53,12){$v_{10}$} \put(58,12){$v_{11}$} \put(63,12){$v_{12}$}
\put(68,12){$v_{13}$} \put(100,40){\circle*{1}}
\put(90,35){\circle*{1}} \put(110,35){\circle*{1}}
\put(100,30){\circle*{1}} \put(100,25){\circle*{1}}
\put(90,20){\circle*{1}} \put(110,20){\circle*{1}}
\put(85,15){\circle*{1}} \put(90,15){\circle*{1}}
\put(95,15){\circle*{1}} \put(100,15){\circle*{1}}
\put(105,15){\circle*{1}} \put(110,15){\circle*{1}}
\put(115,15){\circle*{1}} \thicklines \put(100,40){\line(-2,-1){10}}
\put(100,40){\line(2,-1){10}} \put(90,35){\line(2,-1){10}}
\put(110,35){\line(-2,-1){10}} \put(100,30){\line(0,-1){5}}
\put(100,25){\line(-2,-1){10}} \put(100,25){\line(2,-1){10}}
\put(90,20){\line(-1,-1){5}} \put(90,20){\line(0,-1){5}}
\put(90,20){\line(1,-1){5}} \put(110,20){\line(-2,-1){10}}
\put(110,20){\line(-1,-1){5}} \put(110,20){\line(0,-1){5}}
\put(110,20){\line(1,-1){5}} \put(100,42){$v_{0}$}
\put(85,35){$v_{1}$} \put(111,35){$v_{2}$} \put(103,29){$v_{3}$}
\put(103,25){$v_{4}$} \put(85,20){$v_{5}$} \put(113,20){$v_{6}$}
\put(81,12){$v_{7}$} \put(88,12){$v_{8}$} \put(93,12){$v_{9}$}
\put(99,12){$v_{10}$} \put(104,12){$v_{11}$} \put(110,12){$v_{12}$}
\put(116,12){$v_{13}$} \put(55,3){\small{Fig.3 \ \ \ \
$\displaystyle G_1 \ \hbox{and} \  G_2 \  \
 $}}
\end{picture}

\noindent Then  $\lambda (G_1)= 0.1017 < \lambda(G_2)= 0.1227$. So
the graphs with Faber-Krahn property in $\mathcal{U}_{\pi_1}$ may
not be ball approximation. Moreover, Corollary~\ref{corollary4.4}
does not generally
 hold, since degrees of the interior vertices in $G_1$  do not
satisfy that $v_2\prec v_3$ implies $d(v_2)\le d(v_3)$ for interior
vertices $v_2, v_3$.

\begin{example}\label{example5.2}
Let  $G_3$ and  $G_4$ be the following two graphs
 with degree sequence $\pi_2 =(2, 2, 2, $ $4, 4, 5, $ $ 1, 1, 1, 1, 1, 1, 1)$
 \end{example}

 \setlength{\unitlength}{1mm}
\begin{picture}(120, 45)
\centering \put(40,40){\circle*{1}} \put(30,35){\circle*{1}}
\put(50,35){\circle*{1}}

\put(40,30){\circle*{1}} \put(30,25){\circle*{1}}
\put(50,25){\circle*{1}}

\put(25,20){\circle*{1}} \put(30,20){\circle*{1}}
\put(35,20){\circle*{1}}

\put(45,20){\circle*{1}} \put(55,20){\circle*{1}}
\put(50,20){\circle*{1}} \put(60,20){\circle*{1}}

\thicklines \put(40,40){\line(-2,-1){10}}
\put(40,40){\line(2,-1){10}} \put(30,35){\line(2,-1){10}}
\put(50,35){\line(-2,-1){10}} \put(40,30){\line(-2,-1){10}}
\put(50,25){\line(2,-1){10}} \put(40,30){\line(2,-1){10}}

\put(30,25){\line(-1,-1){5}} \put(30,25){\line(0,-1){5}}
\put(30,25){\line(1,-1){5}} \put(50,25){\line(-1,-1){5}}
\put(50,25){\line(0,-1){5}} \put(50,25){\line(1,-1){5}}
\put(50,25){\line(2,-1){5}} \put(40,42){$v_{0}$}
\put(25,35){$v_{1}$} \put(52,35){$v_{2}$} \put(43,30){$v_{3}$}
\put(24,25){$v_{4}$} \put(53,25){$v_{5}$} \put(23,17){$v_{6}$}
\put(28,17){$v_{7}$} \put(35,17){$v_{8}$} \put(43,17){$v_{9}$}
\put(48,17){$v_{10}$} \put(55,17){$v_{11}$} \put(60,17){$v_{12}$}
\put(100,40){\circle*{1}} \put(95,35){\circle*{1}}
\put(110,35){\circle*{1}} \put(100,30){\circle*{1}}
\put(115,30){\circle*{1}} \put(100,25){\circle*{1}}
\put(110,25){\circle*{1}} \put(115,25){\circle*{1}}
\put(120,25){\circle*{1}} \put(105,20){\circle*{1}}
\put(100,20){\circle*{1}} \put(95,20){\circle*{1}}
\put(90,20){\circle*{1}} \thicklines \put(100,40){\line(-1,-1){5}}
\put(100,40){\line(2,-1){10}} \put(95,35){\line(1,0){15}}
\put(110,35){\line(-2,-1){10}} \put(110,35){\line(1,-1){5}}
\put(100,30){\line(0,-1){5}} \put(100,25){\line(-2,-1){10}}
\put(100,25){\line(1,-1){5}} \put(100,25){\line(-1,-1){5}}
\put(100,25){\line(0,-1){5}} \put(115,30){\line(1,-1){5}}
\put(115,30){\line(0,-1){5}} \put(115,30){\line(-1,-1){5}}
\put(100,42){$v_{0}$} \put(90,35){$v_{1}$} \put(111,35){$v_{2}$}
\put(94,30){$v_{3}$} \put(118,30){$v_{4}$} \put(94,25){$v_{5}$}
\put(107,23){$v_{6}$} \put(113,23){$v_{7}$} \put(118,23){$v_{8}$}
\put(88,17){$v_{9}$} \put(93,17){$v_{10}$} \put(98,17){$v_{11}$}
\put(103,17){$v_{12}$} \put(55,5){\small{Fig.4 \ \ \ $\displaystyle
G_3 \ \hbox{and} \  G_4 $}}
\end{picture}

Then $\lambda (G_3)= 0.2479< \lambda(G_4)= 0.2819$. Hence the graph
with Faber-Krahn property in $\mathcal{U}_{\pi_2}$ may not contain a
triangle. In order to propose our question, we need the following
notation.

Let $\pi = (d_0, d_1, \cdots, d_{n-1})$ be a graphic unicyclic
degree sequence with $2 \leq d_0 \leq d_1\leq \cdots \leq d_{k-1}$
and $d_k=d_{k+1}=\cdots =d_{n-1}=1$.  If $d_2\ge 3,$ then we
construct the graph $U_{\pi}^*$ by the method in section 4. If
$d_0=\cdots=d_{m-1}=2$ and $d_{m}=3$ for $3\le m\le k-1$, we can
construct the graph  $U_{\pi}^{(1)}$ by the similar methods in
section 4, such that $d(v_{0, 1})=d(v_{1, 1})=2,$ $d(v_{1, 2})=3$,
$d(v_{2, 1})=2$, etc. (for example, see $G_1$  in Fig. 3). If
$d_0=\cdots=d_{m-1}=2$ and $d_{m}\ge 4$ for $3\le m\le k-1$, we can
construct the graph $U_{\pi}^{(2)}$ as follows: Let
$\pi^{\prime}=(d_m-2, \cdots, d_{k-1}, 1, \cdots, 1)$ be the
positive integer sequence obtained from $\pi$ by dropping the first
$m$ terms and changing its $(m+1)$-th term  to $d_m-2$. It is easy
to see that $\pi^{\prime}$ is a graphic tree degree sequence. Then
we can get the unique SLO$^\ast$- tree $T_{\pi^{\prime}}$ (see
\cite{biykoglu2007}). Let $U_{\pi}^{(2)}$ be the graph obtained by
identifying a vertex of a cycle of order $m+1$ with the root of
$T_{\pi^{\prime}}$ (for example, see $G_3$ in Fig. 4).

We conclude this paper with the following conjecture.

\begin{conjecture}\label{conjecture5.3}
Let $\pi=(d_0, d_1, \cdots, d_{k-1}, 1, \cdots, 1)$ be a graphic unicyclic
 degree sequence
 with $2 \leq d_0 \leq d_1\leq
\cdots \leq d_{k-1}$ and $d_k=\cdots=d_{n-1}=1$. Then

(1).
 $U_{\pi}^*$ is the unique graph with the Faber-Krahn property in $\mathcal{U}_{\pi}$ if $d_0=2$ and $d_2\ge 3$;

(2).
$U_{\pi}^{(1)}$ is the unique graph  with the Faber-Krahn property in $\mathcal{U}_{\pi}$ if $d_0=\cdots=d_{m-1}=2$ and $d_{m}=3$, where $3\le m\le k-1$;

(3).
$U_{\pi}^{(2)}$ is the unique graph  with the Faber-Krahn property in $\mathcal{U}_{\pi}$ if $d_0=\cdots=d_{m-1}=2$ and $d_{m}\ge 4$, where $3\le m\le k-1$.
\end{conjecture}

\begin{center}
\vskip 0.3cm
 {\bf Acknowledgement}
\end{center}

 The authors would like to thank the  anonymous referees  for
  their kind comments and suggestions.

\end {document}